\documentclass[11pt]{amsart}
\usepackage[leqno]{amsmath}
\usepackage{amssymb,latexsym,soul,cite,amsthm,color,enumitem,graphicx,mathtools,microtype,accents}
\usepackage[colorlinks=true,urlcolor=darkolivegreen,citecolor=darkolivegreen,linkcolor=darkolivegreen,linktocpage,pdfpagelabels,bookmarksnumbered,bookmarksopen]{hyperref}
\definecolor{darkolivegreen}{rgb}{0.33, 0.42, 0.18}
\usepackage[english]{babel}
\usepackage[left=2.5cm,right=2.5cm,top=2.5cm,bottom=2.5cm]{geometry}

\numberwithin{equation}{section}
\newtheorem{theorem}{Theorem}[section]
\theoremstyle{plain}
\newtheorem{lemma}[theorem]{Lemma}
\theoremstyle{plain}

\theoremstyle{plain}

\theoremstyle{definition}

\newcommand{\N}{{\mathbb N}}

\newcommand{\R}{{\mathbb R}}
\newcommand{\eps}{\varepsilon}
\newcommand{\beq}{\begin{equation}}
\newcommand{\eeq}{\end{equation}}
\renewcommand{\le}{\leqslant}
\renewcommand{\ge}{\geqslant}

\newcommand{\lk}{\mathcal{L}_K}
\newcommand{\g}{\mathcal{G}}
\newcommand{\gc}{\overline{\mathcal{G}}}
\newcommand{\wc}{\overset{*}{\rightharpoonup}}
\newcommand{\w}{W^{s,p}_0(\Omega)}

\def\XXint#1#2#3{{\setbox0=\hbox{$#1{#2#3}{\int}$ }
\vcenter{\hbox{$#2#3$ }}\kern-.6\wd0}}

\makeatletter
\newcommand{\leqnomode}{\tagsleft@true}
\newcommand{\reqnomode}{\tagsleft@false}
\makeatother

\newenvironment{enumroman}{\begin{enumerate}

}{\end{enumerate}}

\title[Optimization for fractional $p$-Laplacian]{Optimization problems in rearrangement classes for fractional $p$-Laplacian equations}

\author[A.\ Iannizzotto, G.\ Porru]{Antonio Iannizzotto, Giovanni Porru}

\address[]{Dipartimento di Matematica e Informatica
\newline\indent
Universit\`a degli Studi di Cagliari
\newline\indent
Via Ospedale 72, 09124 Cagliari, Italy}
\email{antonio.iannizzotto@unica.it}
\email{gporru856@gmail.com}


\begin{document}

\begin{abstract}
We discuss two optimization problems related to the fractional $p$-Laplacian. First, we prove the existence of at least one minimizer for the principal eigenvalue of the fractional $p$-Laplacian with Dirichlet conditions, with a bounded weight function varying in a rearrangement class. Then, we investigate the maximization of the energy functional for general nonlinear equations driven by the same operator, as the reaction varies in a rearrangement class. In both cases, we provide a pointwise relation between the optimizing datum and the corresponding solution.
\end{abstract}

\maketitle

\begin{center}
Version of \today\
\end{center}

\section{Introduction}\label{sec1}

\noindent
The present paper deals with some optimization problems related to elliptic equations of nonlinear, nonlocal type, with data varying in rearrangement classes. For the reader's convenience, we recall here the basic definition, referring to Section \ref{sec2} for details. Given a bounded smooth domain $\Omega\subset\R^N$ and a non-negative function $g_0\in L^\infty(\Omega)$, we say that $g\in L^\infty(\Omega)$ lies in the rearrangement class of $g_0$, denoted $\g$, if for all $t\ge 0$
\[\big|\{g>t\}\big| = \big|\{g_0>t\}\big|,\]
where we denote $|\,\cdot\,|$ the $N$-dimensional Lebesgue measure of sets. We may define several functionals $\Phi:\g\to\R$ corresponding to variational problems, and study the optimization problems
\[\min_{g\in\g}\,\Phi(g), \quad \max_{g\in\g}\,\Phi(g).\]
We note that, since $\g$ is not a convex set, the problems above do not fall in the familiar case of convex optimization, whatever the nature of $\Phi$. The following is a classical example. For all $g\in\g$ consider the Dirichlet problem
\[\begin{cases}
-\Delta u = g(x) & \text{in $\Omega$} \\
u = 0 & \text{on $\partial\Omega$,}
\end{cases}\]
which, by classical results in the calculus of variations, admits a unique weak solution $u_g\in H^1_0(\Omega)$. So set
\[\Phi(g) = \int_\Omega gu_g\,dx.\]
The existence of a maximizer for $\Phi$, i.e., of a datum $\hat g\in\g$ s.t.\ for all $g\in\g$
\[\Phi(\hat g) \ge \Phi(g),\]
was proved in \cite{B,B1}, while the existence of a minimizer was investigated in \cite{BML}. One challenging feature of such problem is that, in general, the functional $\Phi$ turns out to be continuous (in a suitable sense) but the class $\g$ fails to be compact. Therefore, a possible strategy consists in optimizing $\Phi$ over the closure $\gc$ of $\g$ in the sequential weak* topology of $L^\infty(\Omega)$ (a much larger, and convex, set), and then proving that the maximizers and minimizers actually lie in $\g$ (which is far from being trivial).
\vskip2pt
\noindent
In addition, due to the nature of the rearrangement equivalence and some functional inequalities, the maximizer/minimizer $g$ may show some structural connection to the solution $u_g$ of the corresponding variational problem, e.g., $g=\eta\circ u_g$ for some monotone mapping $\eta$. This has interesting consequences, for instance let $g_0$ be the characteristic function of some subdomain $D_0\subset\Omega$, then the optimal $g$ is as well the characteristic function of some $D\subset\Omega$ s.t.\ $|D|=|D_0|$. Moreover, since $g=\eta\circ u_g$ and $u_g$ satisfies the Dirichlet condition, we deduce that any optimal domain $D$ has a positive distance from $\partial\Omega$.
\vskip2pt
\noindent
A similar approach applies to several variational problems and functionals. For instance, in \cite{CPS} the authors consider the following $p$-Laplacian equation with $p>1$, $q\in[0,p)$:
\[\begin{cases}
-\Delta_p u = g(x)u^{q-1} & \text{in $\Omega$} \\
u = 0 & \text{on $\partial\Omega$,}
\end{cases}\]
which admits a unique non-negative solution $u_g\in W^{1,p}_0(\Omega)$, and study the maximum and minimum over $\g$ of the functional
\[\Phi(g) = \int_\Omega gu_g^q\,dx.\]
In \cite{CEP}, the following weighted eigenvalue problem is considered:
\[\begin{cases}
-\Delta_p u = \lambda g(x)|u|^{p-2}u & \text{in $\Omega$} \\
u = 0 & \text{on $\partial\Omega$.}
\end{cases}\]
It is well known that the problem above admits a principal eigenvalue $\lambda(g)>0$ (see \cite{KLP}), and the authors prove that $\lambda(g)$ has a minimizer in $\g$.
\vskip2pt
\noindent
In recent years, several researchers have studied optimization problems related to elliptic equations of fractional order (see \cite{MBRS} for a general introduction to such problems and the related variational methods). In the linear framework, the model operator is the $s$-fractional Laplacian with $s\in(0,1)$, defined by
\[(-\Delta)^su(x) = C_{N,s}\lim_{\eps\to 0^+}\int_{B_\eps^c(x)}\frac{u(x)-u(y)}{|x-y|^{N+2s}}\,dy,\]
where $C_{N,s}>0$ is a normalization constant. In \cite{QHZ}, the following problem is examined:
\[\begin{cases}
(-\Delta)^s u+h(x,u) = g(x) & \text{in $\Omega$} \\
u = 0 & \text{in $\Omega^c$,}
\end{cases}\]
where $h(x,\cdot)$ is nondecreasing and grows sublinearly in the second variable. The solution $u_g\in H^s_0(\Omega)$ is unique, so the authors define the energy functional
\[\Phi(g) = \int_\Omega[gu-H(x,u)]\,dx-\frac{1}{2}\iint_{\R^N\times\R^N}\frac{|u(x)-u(y)|^2}{|x-y|^{N+2s}}\,dx\,dy,\]
and investigate the maximum of $\Phi$ over $\g$. Besides, in \cite{ACF}, the existence of a minimizer $g\in\g$ for the principal eigenvalue $\lambda(g)$ of the following problem was proved:
\[\begin{cases}
(-\Delta)^s u = \lambda g(x)u & \text{in $\Omega$} \\
u = 0 & \text{in $\Omega^c$.}
\end{cases}\]
Optimization of the principal eigenvalue of fractional operators has significant applications in biomathematics, see \cite{PV}. In all the aforementioned problems, optimization in $\g$ also yields representation formulas and qualitative properties (e.g., Steiner symmetry over convenient domains) of the optimal data.
\vskip2pt
\noindent
In the present paper, we focus on the following nonlinear, nonlocal operator:
\[\lk u(x) = \lim_{\eps\to 0^+}\int_{B_\eps^c(x)}|u(x)-u(y)|^{p-2}(u(x)-u(y))K(x,y)\,dy.\]
Here $N\ge 2$, $p>1$, $s\in(0,1)$, and $K:\R^N\times\R^N\to\R$ is a measurable kernel s.t.\ for a.e.\ $x,y\in\R^N$
\begin{itemize}[leftmargin=1cm]
\item[$(K_1)$] $K(x,y)=K(y,x)$;
\item[$(K_2)$] $C_1\le K(x,y)|x-y|^{N+ps}\le C_2$ ($0<C_1\le C_2$).
\end{itemize}
If $C_1=C_2=C_{N,p,s}>0$ (a normalization constant varying from one reference to the other), $\lk$ reduces to the $s$-fractional $p$-Laplacian
\[(-\Delta)_p^s u(x) = C_{N,p,s}\lim_{\eps\to 0^+}\int_{B_\eps^c(x)}\frac{|u(x)-u(y)|^{p-2}(u(x)-u(y))}{|x-y|^{N+ps}}\,dy,\]
which in turn coincides with the $s$-fractional Laplacian seen above for $p=2$. The nonlinear operator $\lk$ arises from problems in game theory (see \cite{BCF,C}). Besides, the special case $(-\Delta)_p^s$ can be seen as either an approximation of the classical $p$-Laplace operator for fixed $p$ and $s\to 1^-$ (see \cite{IN}), or an approximation of the fractional $\infty$-Laplacian for fixed $s$ and $p\to\infty$, with applications to the problem of H\"older continuous extensions of functions (see \cite{LL}). Equations driven by the fractional $p$-Laplacian are the subject of a vast literature, dealing with existence, qualitative properties, and regularity of the solutions (see for instance \cite{ILPS,IM,P}).
\vskip2pt
\noindent
Inspired by the cited references, we will examine two variational problems driven by $\lk$, set on a bounded domain $\Omega$ with $C^{1,1}$-smooth boundary, with a datum $g$ varying in a rearrangement class $\g$, and optimize the corresponding functionals. More precisely, we first deal with the nonlinear eigenvalue problem
\beq\label{epg}
\begin{cases}
\lk u = \lambda g(x)|u|^{p-2}u & \text{in $\Omega$} \\
u = 0 & \text{in $\Omega^c$,}
\end{cases}
\eeq
proving the existence and representation of a minimizer of the principal eigenvalue $\lambda(g)$ in $\g$. Then, we consider the general Dirichlet problem
\beq\label{dpg}
\begin{cases}
\lk u+h(x,u) = g(x) & \text{in $\Omega$} \\
u = 0 & \text{in $\Omega^c$,}
\end{cases}
\eeq
and find a maximizer of the associated energy functional. The dual problems (i.e., maximizing the principal eigenvalue and minimizing the energy functional) remain open for now.
\vskip2pt
\noindent
The structure of the paper is the following: in Section \ref{sec2} we recall some preliminaries on rearrangement classes and fractional order equations; in Section \ref{sec3} we deal with the eigenvalue problem \eqref{epg}; and in Section \ref{sec4} we deal with the general Dirichlet problem \eqref{dpg}.
\vskip4pt
\noindent
{\bf Notation.} For all $\Omega\subset\R^N$, we denote by $|\Omega|$ the $N$-dimensional Lebesgue measure of $\Omega$ and $\Omega^c=\R^N\setminus\Omega$. For all $x\in\R^N$, $r>0$ we denote by $B_r(x)$ the open ball centered at $x$ with radius $r$. When we say that $g\ge 0$ in $\Omega$, we mean $g(x)\ge 0$ for a.e.\ $x\in\Omega$, and similar expressions. Whenever $X$ is a function space on the domain $\Omega$, $X_+$ denotes the positive order cone of $X$. In any Banach space we denote by $\to$ strong (or norm) convergence, by $\rightharpoonup$ weak convergence, and by $\wc$ weak* convergence. For all $q\in[1,\infty]$, we denote by $\|\cdot\|_q$ the norm of $L^q(\Omega)$.
Finally, $C$ denotes several positive constants, varying from line to line.

\section{Preliminaries}\label{sec2}

\noindent
In this section we collect some necessary preliminary results on rearrangement classes and on fractional Sobolev spaces.

\subsection{Rearrangement classes}\label{ss21}

Let $\Omega\subset\R^N$ ($N\ge 2$) be a bounded domain, $g_0\in L^\infty(\Omega)$ be s.t.\ $0\le g_0\le M$ in $\Omega$ ($M>0$), and $g_0>0$ on some subset of $\Omega$ with positive measure. We say that a function $g\in L^\infty(\Omega)$ is a rearrangement of $g_0$, denoted $g\sim g_0$, if for all $t\ge 0$
\[\big|\{x\in\Omega:\,g(x)>t\}\big| = \big|\{x\in\Omega:\,g_0(x)>t\}\big|.\]
Also, we define the rearrangement class
\[\g = \big\{g\in L^\infty(\Omega):\,g\sim g_0\big\}.\]
Clearly, $0\le g\le M$ in $\Omega$ for all $g\in\g$. Recalling that $L^\infty(\Omega)$ is the topological dual of $L^1(\Omega)$, we can endow such space with the weak* topology, characterized by the following type of convergence:
\[g_n\wc g \ \Longleftrightarrow \ \lim_n\,\int_\Omega g_nh\,dx = \int_\Omega gh\,dx \ \text{for all $h\in L^1(\Omega)$.}\]
We denote by $\gc$ the closure of $\g$ in $L^\infty(\Omega)$ with respect to such topology. It is proved in \cite{B,B1} that $\gc$ is a sequentially weakly* compact, convex set, and that $0\le g\le M$ in $\Omega$ for all $g\in\gc$. Therefore, given a sequentially weakly* continuous functional $\Phi:\gc\to\R$, there exist $\check g,\hat g\in\gc$ s.t.\ for all $g\in\gc$
\[\Phi(\check g) \le \Phi(g) \le \Phi(\hat g).\]
In general, the extrema are not attained at points of $\g$. As usual, we say that $\Phi$ is G\^ateaux differentiable at $g\in\gc$, if there exists a linear functional $\Phi'(g)\in L^\infty(\Omega)^*$ s.t.\ for all $h\in\gc$
\[\lim_{\tau\to 0^+}\,\frac{\Phi(g+\tau(h-g))-\Phi(g)}{\tau} = \langle\Phi'(g),h-g\rangle.\]
We remark that $g\in\gc$ being a minimizer of $\Phi$ does not imply $\Phi'(g)=0$ in general. Nevertheless, if $\Phi$ is convex, then for all $h\in\gc$
\[\Phi(h) \ge \Phi(g)+\langle\Phi'(g),h-g\rangle,\]
with strict inequality if $\Phi$ is strictly convex and $h\neq g$ (see \cite{R} for an introduction to convex functionals and variational inequalities). Finally, let us recall a technical lemma on optimization of {\em linear} functionals over $\gc$, which also provides a representation formula:

\begin{lemma}\label{lin}
Let $h\in L^1(\Omega)$. Then,
\begin{enumroman}
\item\label{lin1} there exists $\hat g\in\g$ s.t.\ for all $g\in\gc$
\[\int_\Omega \hat g h\,dx \ge \int_\Omega gh\,dx;\]
\item\label{lin2} if $\hat g$ is unique, then there exists a nondecreasing map $\eta:\R\to\R$ s.t.\ $\hat g=\eta\circ h$ in $\Omega$.
\end{enumroman}
\end{lemma}
\begin{proof}
By \cite[Theorems 1, 4]{B} there exists $\hat g\in\g$ which maximizes the linear functional
\[g\mapsto\int_\Omega gh\,dx\]
over $\g$. Given $g\in\gc\setminus\g$, we can find a sequence $(g_n)$ in $\g$ s.t.\ $g_n\wc g$. For all $n\in\N$ we have
\[\int_\Omega\hat gh\,dx \ge \int_\Omega g_nh\,dx,\]
so passing to the limit we get
\[\int_\Omega\hat gh\,dx \ge \int_\Omega gh\,dx,\]
thus proving \ref{lin1}. From \cite[Theorem 5]{B} we have \ref{lin2}.
\end{proof}

\subsection{Fractional Sobolev spaces}\label{ss22}

We recall some basic notions about the variational formulations of problems \eqref{epg}, \eqref{dpg}. For $p>1$, $s\in(0,1)$, all open $\Omega\subseteq\R^N$, and all measurable $u:\Omega\to\R$ we define the Gagliardo seminorm
\[[u]_{s,p,\Omega} = \Big[\iint_{\Omega\times\Omega}\frac{|u(x)-u(y)|^p}{|x-y|^{N+ps}}\,dx\,dy\Big]^\frac{1}{p}.\]
The corresponding fractional Sobolev space is defined by
\[W^{s,p}(\Omega) = \big\{u\in L^p(\Omega):\,[u]_{s,p,\Omega}<\infty\big\}.\]
If $\Omega$ is bounded and with a $C^{1,1}$-smooth boundary, we incorporate the Dirichlet conditions by defining the space
\[\w = \big\{u\in W^{s,p}(\R^N):\,u=0 \ \text{in $\Omega^c$}\big\},\]
endowed with the norm $\|u\|_{\w}=[u]_{s,p,\R^N}$. This is a uniformly convex, separable Banach space with dual $W^{-s,p'}(\Omega)$, s.t.\ $C^\infty_c(\Omega)$ is a dense subset of $\w$, and the embedding $\w\hookrightarrow L^q(\Omega)$ is continuous for all $q\in[1,p^*_s]$ and compact for all $q\in[1,p^*_s)$, where
\[p^*_s = \begin{cases}
\displaystyle\frac{Np}{N-ps} & \text{if $ps<N$} \\
\infty & \text{if $ps\ge N$.}
\end{cases}\]
For a detailed account on fractional Sobolev spaces, we refer the reader to \cite{DNPV,L}. Now let $K:\R^n\times\R^N\to\R$ be a measurable kernel satisfying $(K_1)$, $(K_2)$. We introduce an equivalent norm on $\w$ by setting
\[[u]_K = \Big[\iint_{\R^N\times\R^N}|u(x)-u(y)|^pK(x,y)\,dx\,dy\Big]^\frac{1}{p}.\]
We can now rephrase more carefully the definitions given in Section \ref{sec1}, by defining the operator $\lk:\w\to W^{-s,p'}(\Omega)$ as the gradient of the $C^1$-functional
\[u \mapsto \frac{[u]_K^p}{p}.\]
Equivalently, for all $u,\varphi\in\w$ we set
\[\langle\lk u,\varphi\rangle = \iint_{\R^N\times\R^N}|u(x)-u(y)|^{p-2}(u(x)-u(y))(\varphi(x)-\varphi(y))K(x,y)\,dx\,dy.\]
Both problems that we are going to study belong to the following class of nonlinear, nonlocal Dirichlet problems:
\beq\label{dir}
\begin{cases}
\lk u = f(x,u) & \text{in $\Omega$} \\
u = 0 & \text{in $\Omega^c$,}
\end{cases}
\eeq
where $f:\Omega\times\R\to\R$ is a Carath\'eodory mapping subject to the following subcritical growth conditions: there exist $C>0$, $r\in(1,p^*_s)$ s.t.\ for a.e.\ $x\in\Omega$ and all $t\in\R$
\beq\label{gc}
|f(x,t)| \le C(1+|t|^{r-1}).
\eeq
We say that $u\in\w$ is a weak solution of \eqref{dir}, if for all $\varphi\in\w$
\[\langle\lk u,\varphi\rangle = \int_\Omega f(x,u)\varphi\,dx.\]
There is a wide literature on problem \eqref{dir}, especially for the model case $\lk=(-\Delta)_p^s$, see for instance \cite{DPQ,ILPS,IM1,IMP,P}. We will only need to recall the following properties, which can be proved adapting \cite[Proposition 2.3]{IM1} and \cite[Theorem 1.5]{DPQ}, respectively:

\begin{lemma}\label{apb}
Let $f$ satisfy \eqref{gc}, $u\in\w$ be a weak solution of \eqref{dir}. Then, $u\in L^\infty(\Omega)$.
\end{lemma}

\begin{lemma}\label{smp}
Let $f$ satisfy \eqref{gc}, and $\delta>0$, $c\in C(\overline\Omega)_+$ be s.t.\ for a.e.\ $x\in\Omega$ and all $t\in[0,\delta]$
\[f(x,t) \ge -c(x)t^{p-1}.\]
Also, let $u\in\w_+$ be a weak solution of \eqref{gc}. Then, either $u=0$, or $u>0$ in $\Omega$.
\end{lemma}

\noindent
We conclude with a remark on regularity of the weak solutions. In the model case of the fractional $p$-Laplacian, under hypothesis \eqref{gc}, using Lemma \ref{apb} above and \cite[Theorems 1.1, 2.7]{IM}, it can be seen that whenever $u\in\w$ solves \eqref{dir}, we have $u\in C^s(\R^N)$ and there exist $\alpha\in(0,s)$ depending only on the data of the problem, s.t.\ the function
\[\frac{u}{{\rm dist}(\cdot,\Omega^c)^s}\]
admits a $\alpha$-H\"older continuous extension to $\overline\Omega$. The same result is not known for the general operator $\lk$, except the linear case $p=2$ with a special anisotropic kernel, see \cite{RO}.

\section{Minimization of the principal eigenvalue}\label{sec3}

\noindent
Given a bounded domain $\Omega\subset\R^N$ with $C^{1,1}$-smooth boundary, $p>1$, $s\in(0,1)$, and $K:\R^N\times\R^N\to\R$ satisfying $(K_1)$, $(K_2)$, $g_0\in L^\infty(\Omega)_+$ as in Subsection \ref{ss21}, for any $g\in\g$, $\lambda\in\R$ we consider the nonlinear eigenvalue problem \eqref{epg}. As in Subsection \ref{ss22}, we say that $u\in\w$ is a (weak) solution of \eqref{epg} if for all $\varphi\in\w$
\[\langle\lk u,\varphi\rangle = \lambda\int_\Omega g|u|^{p-2}u\varphi\,dx.\]
We say that $\lambda\in\R$ is an eigenvalue if \eqref{epg} admits a solution $u\neq 0$, which is then a $\lambda$-eigenfunction. Though a full description of the eigenvalues of \eqref{epg} is missing, from \cite{FP,I,LL} we know that for all $g\in L^\infty(\Omega)_+$ there exists a {\em principal eigenvalue} $\lambda(g)>0$, namely the smallest positive eigenvalue, which admits the following variational characterization:
\beq\label{peg}
\lambda(g) = \inf_{u\neq 0}\,\frac{[u]_K^p}{\int_\Omega g|u|^p\,dx}.
\eeq
In addition, from \cite{FP} we know that $\lambda(g)$ is an isolated eigenvalue, simple, with constant sign eigenfunctions, while for any eigenvalue $\lambda>\lambda(g)$ the associated $\lambda$-eigenfunctions change sign in $\Omega$. So, recalling Lemma \ref{smp}, there exists a unique normalized positive $\lambda(g)$-eigenfunction $u_g\in\w$ s.t.\
\[\int_\Omega gu_g^p\,dx = 1, \ [u_g]_K^p = \lambda(g).\]
In particular $g\mapsto\lambda(g)$ defines a real-valued functional defined in the rearrangement class of weights $\g$ (or in $\gc$). The problem we consider in this section is minimization of the principal eigenvalue over $\g$, i.e.,
\[\min_{g\in\g}\,\lambda(g).\]
We prove that such problem admits at least one solution, that any solution actually minimizes $\lambda(g)$ over the larger set $\gc$, while all minimizers over $\gc$ lie in $\g$, and finally that any minimal weight can be represented as a nondecreasing function of the corresponding eigenfunction:

\begin{theorem}\label{min}
Let $\Omega\subset\R^N$ be a bounded domain with $C^{1,1}$-boundary, $p>1$, $s\in(0,1)$, $K:\R^N\times\R^N\to\R$ be measurable satisfying $(K_1)$, $(K_2)$, $g_0\in L^\infty(\Omega)_+\setminus\{0\}$, $\g$ be the rearrangement class of $g_0$, and $\lambda(g)$ be defined by \eqref{peg} for all $g\in\gc$. Then:
\begin{enumroman}
\item\label{min1} there exists $\hat g\in\g$ s.t.\ $\lambda(\hat g)\le\lambda(g)$ for all $g\in\g$;
\item\label{min2} for all $\hat g$ as in \ref{min1} and $g\in\gc\setminus\g$, $\lambda(\hat g)<\lambda(g)$;
\item\label{min3} for all $\hat g$ as in \ref{min1} there exists a nondecreasing map $\eta:\R\to\R$ s.t.\ $\hat g=\eta\circ u_{\hat g}$ in $\Omega$.
\end{enumroman}
\end{theorem}

\noindent
The proof of Theorem \ref{min} requires several steps. First, we aim at minimizing $\lambda(g)$ over $\gc$. This is possible due to the following lemma:

\begin{lemma}\label{esc}
The functional $\lambda(g)$ is sequentially weakly* continuous in $\gc$.
\end{lemma}
\begin{proof}
Let $(g_n)$ be a sequence in $\gc$ s.t.\ $g_n\wc g$, and for simplicity denote $u_n=u_{g_n}$ for all $n\in\N$, and $u=u_g$. We need to prove that $\lambda(g_n)\to\lambda(g)$. Since $0\le g_n,g\le M$ in $\Omega$, and $u^p\in L^1(\Omega)\subset L^\infty(\Omega)^*$, we have
\[\lim_n\,\int_\Omega g_nu^p\,dx = \int_\Omega gu^p\,dx = 1.\]
By \eqref{peg} and normalization, we have for all $n\in\N$
\[\lambda(g_n) = \frac{[u_n]_K^p}{\int_\Omega g_nu_n^p\,dx} \le \frac{[u]_K^p}{\int_\Omega g_nu^p\,dx} = \frac{\lambda(g)}{\int_\Omega g_nu^p\,dx},\]
and the latter tends to $\lambda(g)$ as $n\to\infty$. Hence we have
\beq\label{esc1}
\limsup_n\,\lambda(g_n) \le \lambda(g).
\eeq
By \eqref{esc1}, $(u_n)$ is bounded in $\w$. By reflexivity and the compact embedding $\w\hookrightarrow L^p(\Omega)$, passing to a subsequence we have $u_n\rightharpoonup v$ in $\w$, $u_n\to v$ in $L^p(\Omega)$, and $u_n(x)\to v(x)$ for a.e.\ $x\in\Omega$, as $n\to\infty$. In particular, $v\ge 0$ in $\Omega$. By convexity we have
\[\liminf_n\,[u_n]_K^p \ge [v]_K^p.\]
Besides, by strong convergence in $L^p(\Omega)$ we have
\[\lim_n\,\int_\Omega gu_n^p\,dx = \int_\Omega gv^p\,dx.\]
Therefore, by H\"older's inequality and the uniform bound of $(g_n)$, for all $n\in\N$ we have
\begin{align*}
\Big|\int_\Omega g_n(u_n^p-v^p)\,dx\Big| &\le C\int_\Omega|u_n^p-v^p|\,dx \\
&\le C\int_\Omega(u_n^{p-1}+v^{p-1})|u_n-v|\,dx \\
&\le C\big[\|u_n\|_p^{p-1}+\|v\|_p^{p-1}\big]\|u_n-v\|_p,
\end{align*}
and the latter tends to $0$ as $n\to\infty$. So we have
\[\lim_n\,\int_\Omega g_nu_n^p\,dx = \int_\Omega gv^p\,dx.\]
From the relations above we deduce
\beq\label{esc2}
\liminf_n\,\lambda(g_n) \ge \frac{[v]_K^p}{\int_\Omega gv^p\,dx} \ge \lambda(g).
\eeq
Finally, \eqref{esc1} and \eqref{esc2} imply $\lambda(g_n)\to\lambda(g)$.
\end{proof}

\noindent
As seen in Subsection \ref{ss22}, Lemma \ref{esc} along with compactness of $\gc$, proves that $\lambda(g)$ admits a minimizer and a maximizer in $\gc$. We next need to ensure that at least one minimizer lies in the smaller set $\g$. With this aim in mind, we rephrase the problem as follows. For all $g\in\gc$ set
\[\Phi(g) = \frac{1}{\lambda(g)^2},\]
so that
\[\min_{g\in\g}\,\lambda(g) \ \Longleftrightarrow \ \max_{g\in\g}\,\Phi(g).\]
By Lemma \ref{esc}, the functional $\Phi$ is sequentially weakly* continuous in $\gc$. In the next lemmas we will investigates further properties of $\Phi$:

\begin{lemma}\label{ecv}
The functional $\Phi$ is strictly convex in $\gc$.
\end{lemma}
\begin{proof}
We introduce an alternative expression for $\Phi$. For all $g\in\gc$, $u\in\w_+$ set
\[F(g,u) = 2\int_\Omega gu^p\,dx-[u]_K^{2p}.\]
We fix $g\in\gc$ and maximize $F(g,\cdot)$ over positive functions. For all $u\in\w_+\setminus\{0\}$ and $\tau>0$, the function
\[F(g,\tau u) = 2\tau^p\int_\Omega gu^p\,dx-\tau^{2p}[u]_K^p\]
is differentiable in $\tau$ with derivative
\[\frac{\partial}{\partial\tau}F(g,\tau u) = 2p\tau^{p-1}\int_\Omega gu^p\,dx-2p\tau^{2p-1}[u]_K^{2p}.\]
So the maximum of $\tau\mapsto F(g,\tau u)$ is attained at
\[\tau_0 = \frac{\Big[\int_\Omega gu^p\,dx\Big]^\frac{1}{p}}{[u]_K^2} = \lambda(g)^{-\frac{2}{p}} > 0,\]
and amounts at
\[F(g,\tau_0u) = \frac{\Big[\int_\Omega gu^p\,dx\Big]^2}{[u]_K^{2p}}.\]
Maximizing further over $u$, we obtain
\[\sup_{u>0}\,F(g,u) = \sup_{u\in\w_+\setminus\{0\}}\,\frac{\Big[\int_\Omega gu^p\,dx\Big]^2}{[u]_K^{2p}}.\]
Noting that $[|u|]_K\le[u]_K$ for all $u\in\w$, and recalling \eqref{peg}, we have for all $g\in\gc$
\beq\label{ecv1}
\Phi(g) = \sup_{u\in\w_+\setminus\{0\}}\,F(g,u) = \frac{1}{\lambda(g)^2}.
\eeq
We claim that supremum in \eqref{ecv1} is attained at the unique function
\beq\label{ecv2}
\tilde u_g = \frac{u_g}{\lambda(g)^\frac{2}{p}}.
\eeq
Indeed, by normalization of $u_g$ we have
\[F(g,\tilde u_g) = \frac{2}{\lambda(g)^2}\,\int_\Omega gu_g^p\,dx-\frac{[u_g]_K^{2p}}{\lambda(g)^4} = \frac{1}{\lambda(g)^2}.\]
For uniqueness, first consider a function $u=\tau u_g$ for some $\tau\neq\tau_0$. By the argument above we have
\[F(g,\tau u_g) < F(g,\tau_0 u_g) = F(g,\tilde u_g) = \frac{1}{\lambda(g)^2}.\]
Besides, for all $v\in\w_+\setminus\{0\}$ which is not a $\lambda(g)$-eigenfunction, arguing as above with $v$ replacing $u$, and recalling that the infimum in \eqref{peg} is attained only at principal eigenfunctions, we have
\[F(g,v) \le F(g,\tau_0v) = \frac{\Big[\int_\Omega gv^p\,dx\Big]^2}{[v]_K^{2p}} < \frac{1}{\lambda(g)^2}.\]
So, $\tilde u_g$ is the unique maximizer of \eqref{ecv1}.
\vskip2pt
\noindent
We now prove that $\Phi$ is convex. Let $g_1,g_2\in\gc$, $\tau\in(0,1)$ and set
\[g_\tau = (1-\tau)g_1+\tau g_2,\]
so $g_\tau\in\gc$ (a convex set, as seen in Subsection \ref{ss21}). For all $u\in\w_+\setminus\{0\}$ we have by \eqref{ecv1}
\begin{align*}
F(g_\tau,u) &= 2(1-\tau)\int_\Omega g_1u^p\,dx+2\tau\int_\Omega g_2u^p\,dx-[u]_K^{2p} \\
&= (1-\tau)F(g_1,u)+\tau F(g_2,u) \le (1-\tau)\Phi(g_1)+\tau\Phi(g_2).
\end{align*}
Taking the supremum over $u$ and using \eqref{ecv1} again,
\[\Phi(g_\tau) \le (1-\tau)\Phi(g_1)+\tau\Phi(g_2).\]
To prove that $\Phi$ is {\em strictly} convex, we argue by contradiction, assuming that for some $g_1\neq g_2$ as above and $\tau\in(0,1)$
\[\Phi(g_\tau) = (1-\tau)\Phi(g_1)+\tau\Phi(g_2).\]
Set $\tilde u_i=\tilde u_{g_i}$ ($i=1,2$) and $\tilde u_\tau=\tilde u_{g_\tau}$ for brevity. Then, by \eqref{ecv1} and the equality above
\[(1-\tau)F(g_1,\tilde u_\tau)+\tau F(g_2,\tilde u_\tau) = (1-\tau)F(g_1,\tilde u_1)+\tau F(g_2,\tilde u_2).\]
Recalling that $\tilde u_i$ is the only maximizer of $F(g_i,\cdot)$, the last inequality implies $\tilde u_1=\tilde u_2=\tilde u_\tau$, as well as
\[\Phi(g_1) = F(g_1,\tilde u_\tau) = F(g_2,\tilde u_\tau) = \Phi(g_2).\]
Therefore we have $\lambda(g_1)=\lambda(g_2)=\lambda$. Moreover, $\tilde u_\tau>0$ is a $\lambda$-eigenfunction with both weight $g_1$, $g_2$, i.e., for all $\varphi\in\w$
\[\lambda\int_\Omega g_1\tilde u_\tau^{p-1}\varphi\,dx = \langle\lk\tilde u_\tau,\varphi\rangle = \lambda\int_\Omega g_2\tilde u_\tau^{p-1}\varphi\,dx.\]
So $g_1\tilde u_\tau^{p-1}=g_2\tilde u_\tau^{p-1}$ in $\Omega$, which in turn, since $\tilde u_\tau>0$, implies $g_1=g_2$ a.e.\ in $\Omega$, a contradiction.
\end{proof}

\noindent
The next lemma establishes differentiability of $\Phi$:

\begin{lemma}\label{egd}
The functional $\Phi$ is G\^ateaux differentiable in $\gc$, and for all $g,h\in\gc$
\[\langle\Phi'(g),h-g\rangle = 2\int_\Omega(h-g)\tilde u_g^p\,dx,\]
where $\tilde u_g$ is the principal eigenfunction normalized as in \eqref{ecv2}.
\end{lemma}
\begin{proof}
First, let $(g_n)$ be a sequence in $\gc$ s.t.\ $g_n\wc g$, and set for brevity $\tilde u_n=\tilde u_{g_n}$, $\tilde u=\tilde u_g$. We claim that
\beq\label{egd1}
\lim_n\,\int_\Omega|\tilde u_n-\tilde u|^p\,dx = 0.
\eeq
Indeed, by normalization we have for all $n\in\N$
\[[\tilde u_n]_K^{2p} = \frac{1}{\lambda(g_n)^2},\]
and the latter is bounded from above, since $\lambda(g)$ has a minimizer in $\gc$. So, $(\tilde u_n)$ is bounded in $\w$. By uniform convexity and the compact embedding $\w\hookrightarrow L^p(\Omega)$, passing to a subsequence we have $\tilde u_n\rightharpoonup v$ in $\w$, $\tilde u_n\to v$ in $L^p(\Omega)$, and $\tilde u_n(x)\to v(x)$ for a.e.\ $x\in\Omega$, as $n\to\infty$ (in particular $v\ge 0$ in $\Omega$). As in Lemma \ref{esc} we see that
\[\liminf_n\,[\tilde u_n]_K^{2p} \ge [v]_K^{2p},\]
and
\[\lim_n\,\int_\Omega g_n\tilde u_n^p\,dx = \int_\Omega gv^p\,dx.\]
By Lemma \ref{esc} we have $\Phi(g_n)\to\Phi(g)$, so by \eqref{ecv1} we get
\begin{align*}
\Phi(g) &= \lim_n\,F(g_n,\tilde u_n) \\
&\le 2\lim_n\,\int_\Omega g_n\tilde u_n^p\,dx-\liminf_n\,[\tilde u_n]_K^{2p} \\
&\le 2\int_\Omega gv^p\,dx-[v]_K^{2p} \\
&= F(g,v) \le \Phi(g).
\end{align*}
Therefore $v$ is a maximizer of $F(g,\cdot)$ over $\w_+$, hence by uniqueness $v=\tilde u$. Then we have $\tilde u_n\to\tilde u$ in $L^p(\Omega)$, which is equivalent to \eqref{egd1}.
\vskip2pt
\noindent
We claim that for all $n\in\N$
\beq\label{egd2}
\Phi(g)+2\int_\Omega(g_n-g)\tilde u^p\,dx \le \Phi(g_n) \le \Phi(g)+2\int_\Omega(g_n-g)\tilde u_n^p\,dx.
\eeq
Indeed, by \eqref{ecv1} we have
\begin{align*}
\Phi(g)+2\int_\Omega(g_n-g)\tilde u^p\,dx &\le \Phi(g_n) \\
&= F(g,\tilde u_n)+2\int_\Omega(g_n-g)\tilde u_n^p\,dx \\
&\le \Phi(g)+2\int_\Omega(g_n-g)\tilde u_n^p\,dx.
\end{align*}
Now fix $g,h\in\gc$, $g\neq h$, and a sequence $(\tau_n)$ in $(0,1)$ s.t.\ $\tau_n\to 0$. By convexity of $\gc$, we have for all $n\in\N$
\[g_n = g+\tau_n(h-g) \in \gc.\]
Also, clearly $g_n\wc g$. By \eqref{egd2}, setting as usual $\tilde u_n=\tilde u_{g_n}$ and $\tilde u=\tilde u_g$, we have for all $n\in\N$
\[2\tau_n\int_\Omega(h-g)\tilde u^p\,dx \le \Phi(g_n)-\Phi(g) \le 2\tau_n\int_\Omega(h-g)\tilde u_n^p\,dx.\]
Dividing by $\tau_n>0$ and recalling \eqref{egd1}, we get
\[\lim_n\,\frac{\Phi(g+\tau_n(h-g))-\Phi(g)}{\tau_n} = 2\int_\Omega(h-g)\tilde u^p\,dx.\]
Note that $2\tilde u^p\in L^1(\Omega)\subset L^\infty(\Omega)^*$, and by arbitrariness of the sequence $(\tau_n)$ we deduce that $\Phi$ is G\^ateaux differentiable at $g$ with
\[\langle\Phi'(g),h-g\rangle = 2\int_\Omega(h-g)\tilde u^p\,dx,\]
which concludes the proof.
\end{proof}

\noindent
We can now prove the main result of this section:
\vskip2pt
\noindent
{\em Proof of Theorem \ref{min}.} We already know that $\Phi$ has a maximizer $\bar g$ over $\gc$. Set $\bar w=2\tilde u_{\bar g}^p\in L^1(\Omega)$, then by Lemma \ref{egd} we have $\Phi'(\bar g)=\bar w$. Now we maximize on $\gc$ the linear functional
\[g \mapsto \int_\Omega g\bar w\,dx.\]
By Lemma \ref{lin} \ref{lin1}, there exists $\hat g\in\g$ s.t.\ for all $g\in\gc$
\[\int_\Omega\hat g\bar w\,dx \ge \int_\Omega g\bar w\,dx.\]
In particular we have
\beq\label{min4}
\int_\Omega\hat g\bar w\,dx \ge \int_\Omega \bar g\bar w\,dx.
\eeq
By Lemma \ref{ecv}, the functional $\Phi$ is convex. Therefore, using also Lemma \ref{egd} and \eqref{min4}, we have
\[\Phi(\hat g) \ge \Phi(\bar g)+\int_\Omega(\hat g-\bar g)\bar w\,dx \ge \Phi(\bar g).\]
Thus, $\hat g\in\g$ is as well a maximizer of $\Phi$ over $\gc$, which proves \ref{min1} since maximizers of $\Phi$ and minimizers of $\lambda(g)$ coincide. In addition, by the relation above we have
\[\int_\Omega(\hat g-\bar g)\bar w\,dx = 0.\]
We will now prove that $\hat g=\bar g$, arguing by contradiction. Assume $\hat g\neq\bar g$, then by strict convexity of $\Phi$ (Lemma \ref{ecv} again) we have
\[\Phi(\hat g) > \Phi(\bar g)+\int_\Omega(\hat g-\bar g)\bar w\,dx = \Phi(\bar g),\]
against maximality of $\bar g$. So, any maximizer of $\Phi$ over $\gc$ actually lies in $\g$, which proves \ref{min2}. Finally, let $\hat g\in\g$ be any maximizer of $\Phi$ and set $\hat w=2\tilde u_{\hat g}^p\in L^1(\Omega)$. By Lemmas \ref{ecv} and \ref{egd}, for all $g\in\gc\setminus\{\hat g\}$ we have
\[\Phi(\hat g) \ge \Phi(g) > \Phi(\hat g)+\int_\Omega(g-\hat g)\hat w\,dx,\]
hence
\[\int_\Omega\hat g\hat w\,dx > \int_\Omega g\hat w\,dx.\]
Equivalently, $\hat g$ is the only maximizer over $\gc$ of the linear functional above, induced by the function $\hat w$. By Lemma \ref{lin} \ref{lin2}, there exists a nondecreasing map $\tilde\eta:\R\to\R$ s.t.\ in $\Omega$
\[\hat g = \tilde\eta\circ\hat w.\]
Now we recall \eqref{ecv2} and the definition of $\hat w$, and by setting for all $t\ge 0$
\[\eta(t) = \tilde\eta\Big(\frac{2t^p}{\lambda(\hat g)^2}\Big),\]
while $\eta(t)=\eta(0)$ for all $t<0$, we immediately see that $\eta:\R\to\R$ is a nondecreasing map s.t.\ $\hat g=\eta\circ u_{\hat g}$ in $\Omega$, thus proving \ref{min3}. \qed

\section{Maximization of the energy functional}\label{sec4}

\noindent
In this section we deal with problem \eqref{dpg}. Again, $\Omega\subset\R^N$ is a bounded domain with $C^{1,1}$-smooth boundary, $p>1$, $s\in(0,1)$, and $K:\R^N\times\R^N\to\R$ is a measurable kernel satisfying $(K_1)$, $(K_2)$. Also, $h:\Omega\times\R\to\R_+$ is a Carath\'eodory mapping satisfying the following conditions:
\begin{itemize}[leftmargin=1cm]
\item[$(h_1)$] $h(x,\cdot)$ is nondecreasing in $\R$ for a.e.\ $x\in\Omega$;
\item[$(h_2)$] $h(x,t)\le C_0(1+|t|^{q-1})$ for a.e.\ $x\in\Omega$ and all $t\in\R$, with $C_0>0$ and $q\in(1,p)$.
\end{itemize}
Finally, $g$ lies in $\gc$, where $\g$ is the rearrangement class based on some function $g_0\in L^\infty(\Omega)_+$. As in Subsection \ref{ss22}, a weak solution $u\in\w$ of \eqref{dpg} satisfies for all $\varphi\in\w$
\[\langle\lk(u),\varphi\rangle+\int_\Omega h(x,u)\varphi\,dx = \int_\Omega g\varphi\,dx.\]
By classical results (see for instance \cite{IM2} for the fractional $p$-Laplacian), for all $g\in\gc$ problem \eqref{dpg} has a unique solution $u_g\in\w$. In addition, by Lemma \ref{apb} we have $u_g\in L^\infty(\Omega)$.
\vskip2pt
\noindent
We will now define an energy functional. For all $(x,t)\in\Omega\times\R$ set
\[H(x,t) = \int_0^t h(x,\tau)\,d\tau,\]
and for all $g\in\gc$, $u\in\w$ define
\[E(g,u) = \int_\Omega\big[gu-H(x,u)\big]\,dx-\frac{[u]_K^p}{p}.\]
It is well known that the solution $u_g$ of \eqref{dpg} is the only maximizer of $E(g,\cdot)$ in $\w$. We set
\beq\label{phi}
\Phi(g) = \sup_{u\in\w}\,E(g,u) = E(g,u_g),
\eeq
and we study the maximization problem
\[\max_{g\in\g}\,\Phi(g).\]
We will prove that $\Phi$ admits at least one maximizer in $\g$, which actually maximizes $\Phi$ over $\gc$. Also, all maximizers over $\gc$ lie in $\g$, and any maximizer $g$ can be represented as a nondecreasing function of the corresponding solution $u_g$:

\begin{theorem}\label{max}
Let $\Omega\subset\R^N$ be a bounded domain with $C^{1,1}$-boundary, $p>1$, $s\in(0,1)$, $K:\R^N\times\R^N\to\R$ be measurable satisfying $(K_1)$, $(K_2)$, $g_0\in L^\infty(\Omega)_+\setminus\{0\}$, $\g$ be the rearrangement class of $g_0$, and $\Phi(g)$ be defined by \eqref{phi} for all $g\in\gc$. Then:
\begin{enumroman}
\item\label{max1} there exists $\hat g\in\g$ s.t.\ $\Phi(\hat g)\ge\Phi(g)$ for all $g\in\g$;
\item\label{max2} for all $\hat g$ as in \ref{max1} and $g\in\gc\setminus\g$, $\Phi(\hat g)>\Phi(g)$;
\item\label{max3} for all $\hat g$ as in \ref{max1} there exists a nondecreasing map $\eta:\R\to\R$ s.t.\ $\hat g=\eta\circ u_{\hat g}$ in $\Omega$.
\end{enumroman}
\end{theorem}

\noindent
Theorem \ref{max} is analogous to Theorem \ref{min} above, while in fact the problem is easier since we do not need to consider normalization to ensure uniqueness, unlike in problem \eqref{peg}, and the functional $\Phi$ is dealt with directly. On the other hand, in this case we have no information on the sign of the solution $u_g$.
\vskip2pt
\noindent
Again we examine separately the properties of $\Phi$ in some preparatory lemmas (the proofs, similar to those seen in Section \ref{sec3}, are included for completeness). We begin proving continuity of $\Phi$:

\begin{lemma}\label{dsc}
The functional $\Phi$ is sequentially weakly* continuous in $\gc$.
\end{lemma}
\begin{proof}
Let $(g_n)$ be a sequence in $\gc$ s.t.\ $g_n\wc g$, and denote $u_n=u_{g_n}$, $u=u_g$. By \eqref{phi}, for all $n\in\N$ we have
\begin{align*}
\Phi(g_n) &= E(g_n,u_n) \\
&\ge E(g_n,u) \\
&= E(g,u)+\int_\Omega(g_n-g)u\,dx \\
&= \Phi(g)+\int_\Omega(g_n-g)u\,dx.
\end{align*}
Passing to the limit as $n\to\infty$ and using weak* convergence, we get
\beq\label{dsc1}
\liminf_n\Phi(g_n) \ge \Phi(g)+\lim_n\int_\Omega(g_n-g)u\,dx = \Phi(g).
\eeq
From \eqref{dpg} with datum $g_n$ and solution $u_n$, multiplying by $u_n$ again, we get for all $n\in\N$
\beq\label{dsc2}
\int_\Omega\big[g_n-h(x,u_n)\big]u_n\,dx = [u_n]_K^p.
\eeq
Since $\|g_n\|_\infty\le M$ and by the continuous embedding $\w\hookrightarrow L^1(\Omega)$, we have
\[\Big|\int_\Omega g_nu_n\,dx\Big| \le C[u_n]_K,\]
with $C>0$ independent of $n$. Also, by $(h_2)$ and the continuous embedding $\w\hookrightarrow L^q(\Omega)$, we have
\[\Big|\int_\Omega h(x,u_n)u_n\,dx\Big| \le C\int_\Omega\big[|u_n|+|u_n|^q\big]\,dx \le C[u_n]_K+C[u_n]_K^q.\]
So \eqref{dsc2} implies for all $n\in\N$
\[[u_n]_K^p \le C[u_n]_K+C[u_n]_K^q.\]
Recalling that $q<p$, we deduce that $(u_n)$ is bounded in $\w$. Passing to a subsequence, we have $u_n\rightharpoonup v$ in $\w$, $u_n\to v$ in $L^p(\Omega)$, and $u_n(x)\to v(x)$ for a.e.\ $x\in\Omega$, as $n\to\infty$. By convexity we have
\[\liminf_n\,[u_n]_K^p \ge [v]_K^p.\]
Reasoning as in Lemma \ref{esc} we find
\beq\label{dsc3}
\lim_n\,\int_\Omega g_nu_n\,dx = \int_\Omega gv\,dx.
\eeq
Finally, we have
\beq\label{dsc4}
\lim_n\,\int_\Omega H(x,u_n)\,dx = \int_\Omega H(x,v)\,dx.
\eeq
Indeed, applying $(h_2)$, Lagrange's rule, and H\"older's inequality, we get for all $n\in\N$
\begin{align*}
\int_\Omega\big|H(x,u_n)-H(x,v)\big|\,dx &\le C\int_\Omega\big[1+|u_n|^{q-1}+|v|^{q-1}\big]|u_n-v|\,dx \\
&\le C\|u_n-v\|_1+C\big[\|u_n\|_q^{q-1}+\|v\|_q^{q-1}\big]\|u_n-v\|_q,
\end{align*}
and the latter tends to $0$ as $n\to\infty$, by the continuous embeddings of $L^p(\Omega)$ into $L^1(\Omega)$, $L^q(\Omega)$, respectively, thus proving \eqref{dsc4}.
\vskip2pt
\noindent
Next, we start from \eqref{phi} and we apply \eqref{dsc3} and \eqref{dsc4}:
\begin{align*}
\limsup_n\Phi(g_n) &= \limsup_n E(g_n,u_n) \\
&\le \lim_n\,\int_\Omega[g_nu_n-H(x,u_n)\big]\,dx-\liminf_n\frac{[u_n]_K^p}{p} \\
&\le \int_\Omega\big[gv-H(x,v)\big]\,dx-\frac{[v]_K^p}{p} = E(g,v),
\end{align*}
and the latter does not exceed $\Phi(g)$, so
\beq\label{dsc5}
\limsup_n\Phi(g_n) \le \Phi(g).
\eeq
Comparing \eqref{dsc1} and \eqref{dsc5}, we have $\Phi(g_n)\to\Phi(g)$, which concludes the proof.
\end{proof}

\noindent
By Lemma \ref{dsc}, $\Phi$ has both a minimizer and a maximizer over $\gc$. Next we prove strict convexity:

\begin{lemma}\label{dcv}
The functional $\Phi$ is strictly convex in $\gc$.
\end{lemma}
\begin{proof}
Convexity of $\Phi$ follows as in Lemma \ref{ecv}. To prove strict convexity, we argue by contradiction. Let $g_1,g_2\in\gc$ be s.t.\ $g_1\neq g_2$, set for all $\tau\in(0,1)$
\[g_\tau = (1-\tau)g_1+\tau g_2 \in \gc,\]
and assume that for some $\tau\in(0,1)$
\[\Phi(g_\tau) = (1-\tau)\Phi(g_1)+\tau\Phi(g_2).\]
As usual, set $u_i=u_{g_i}$ ($i=1,2$) and $u_\tau=u_{g_\tau}$. By linearity of $E(g,u_\tau)$ in $g$ and \eqref{phi}, the relation above rephrases as
\[(1-\tau)E(g_1,u_\tau)+\tau E(g_2,u_\tau) = E(g_\tau,u_\tau) = (1-\tau)E(g_1,u_1)+\tau E(g_2,u_2).\]
Recalling that $E(g_i,u_\tau)\le E(g_i,u_i)$ ($i=1,2$) and the uniqueness of the maximizer in \eqref{phi}, we deduce $u_1=u_2=u_\tau$. Now test \eqref{dpg} with an arbitrary $\varphi\in\w$:
\[\int_\Omega g_1\varphi\,dx = \langle\lk(u_\tau),\varphi\rangle+\int_\Omega h(x,u_\tau)\varphi\,dx = \int_\Omega g_2\varphi\,dx.\]
So we have $g_1=g_2$ a.e.\ in $\Omega$, a contradiction. Thus, $\Phi$ is strictly convex.
\end{proof}

\noindent
The last property we need is differentiability:

\begin{lemma}\label{dgd}
The functional $\Phi$ is G\^ateaux differentiable in $\gc$, and for all $g,k\in\gc$
\[\langle\Phi'(g),k-g\rangle = \int_\Omega (k-g)u_g\,dx.\]
\end{lemma}
\begin{proof}
Let $(g_n)$ be a sequence in $\gc$ s.t.\ $g_n\wc g$, and let $u_n=u_{g_n}$, $u=u_g$. From Lemma \ref{dsc} we know that $\Phi(g_n)$ tends to $\Phi(g)$, i.e.,
\[\lim_n E(g_n,u_n) = E(g,u).\]
We further claim that
\beq\label{dgd1}
\lim_n\int_\Omega|u_n-u|^p\,dx = 0.
\eeq
Indeed, reasoning as in Lemma \ref{dsc} we see that $(u_n)$ is bounded in $\w$. So, passing to a subsequence, we have $u_n\rightharpoonup v$ in $\w$, $u_n\to v$ in $L^p(\Omega)$, and $u_n(x)\to v(x)$ for a.e.\ $x\in\Omega$, as $n\to\infty$. Again as in Lemma \ref{dsc} it is seen that
\[\liminf_n[u_n]_K^p \ge [v]_K^p\]
and
\[\lim_n\int_\Omega\big[g_nu_n-H(x,u_n)\big]\,dx = \int_\Omega\big[gv-H(x,v)\big]\,dx.\]
Summing up, we get
\[E(g,u) = \lim_n E(g_n,u_n) \le E(g,v),\]
which implies $u=v$ by uniqueness of the maximizer in \eqref{phi}. So $u_n\to u$ in $L^p(\Omega)$, which yields \eqref{dgd1}. In addition, reasoning as in Lemma \ref{egd}, for all $n\in\N$ we have
\beq\label{dgd2}
\Phi(g)+\int_\Omega (g_n-g)u\,dx \le \Phi(g_n) \le \Phi(g)+\int_\Omega(g_n-g)u_n\,dx.
\eeq
Now fix $k\in\gc\setminus\{g\}$ and a sequence $(\tau_n)$ in $(0,1)$ s.t.\ $\tau_n\to 0$ as $n\to\infty$. Set
\[g_n = g+\tau_n(k-g) \in \gc,\]
so that $g_n\wc g$. By \eqref{dgd2} with such choice of $g_n$, we have for all $n\in\N$
\[\int_\Omega(k-g)u\,dx \le \frac{\Phi(g+\tau_n(k-g))-\Phi(g)}{\tau_n} \le \int_\Omega(k-g)u_n\,dx.\]
Passing to the limit for $n\to\infty$, and noting that by \eqref{dgd1} we have in particular $u_n\to u$ in $L^1(\Omega)$, we get
\[\lim_n \frac{\Phi(g+\tau_n(k-g))-\Phi(g)}{\tau_n} = \int_\Omega(k-g)u\,dx.\]
By arbitrariness of $(\tau_n)$, and noting that $u\in L^1(\Omega)\subset L^\infty(\Omega)^*$, we see that $\Phi$ is G\^ateaux differentiable at $g$ with
\[\langle\Phi'(g),k-g\rangle = \int_\Omega(k-g)u\,dx,\]
which concludes the proof.
\end{proof}

\noindent
We can now prove our optimization result, with a similar argument as in Section \ref{sec3}:
\vskip4pt
\noindent
{\em Proof of Theorem \ref{max}.} By Lemma \ref{dsc} and sequential weak* compactness of $\gc$, there exists $\bar g\in\gc$ s.t.\ for all $g\in\gc$
\[\Phi(\bar g) \ge \Phi(g).\]
Set $\bar u=u_{\bar g}\in\w$, then by Lemma \ref{dgd} we have for all $k\in\gc\setminus\{\bar g\}$
\[\langle\Phi'(\bar g),k-\bar g\rangle = \int_\Omega(k-\bar g)\bar u\,dx.\]
Since $\bar u\in L^1(\Omega)$, by Lemma \ref{lin} \ref{lin1} there exists $\hat g\in\g$ s.t.\ for all $g\in\gc$
\[\int_\Omega\hat g\bar u\,dx \ge \int_\Omega g\bar u\,dx,\]
in particular
\beq\label{max4}
\int_\Omega\hat g\bar u\,dx \ge \int_\Omega\bar g\bar u\,dx.
\eeq
By convexity of $\Phi$ (Lemma \ref{dcv}) and \eqref{max4}, we have
\[\Phi(\hat g) \ge \Phi(\bar g)+\int_\Omega(\hat g-\bar g)\bar u\,dx \ge \Phi(\bar g).\]
Therefore, $\hat g\in\g$ is a maximizer of $\Phi$ over $\gc$, which proves \ref{max1}. In fact we have $\hat g=\bar g$, otherwise by strict convexity (Lemma \ref{dcv} again) and \eqref{max4} we would have
\[\Phi(\hat g) > \Phi(\bar g)+\int_\Omega(\hat g-\bar g)\bar u\,dx \ge \Phi(\bar g),\]
against maximality of $\bar g$. Thus, any maximizer of $\Phi$ over $\gc$ actually lies in $\g$, which proves \ref{max2}. Finally, let $\hat g\in\g$ be a maximizer of $\Phi$ and set $\hat u=u_{\hat g}\in\w$. As we have seen before, $\hat g$ is the only maximizer in $\gc$ for the linear functional
\[g \mapsto \int_\Omega g\hat u\,dx,\]
hence by Lemma \ref{lin} \ref{lin2} there exists a nondecreasing map $\eta:\R\to\R$ s.t.\ $\hat g=\eta\circ\hat u$ in $\Omega$, thus proving \ref{max3}. \qed
\vskip4pt
\noindent
{\bf Acknowledgement.} The first author is a member of GNAMPA (Gruppo Nazionale per l'Analisi Matematica, la Probabilit\`a e le loro Applicazioni) of INdAM (Istituto Nazionale di Alta Matematica 'Francesco Severi'), and is partially supported by the research project {\em Problemi non locali di tipo stazionario ed evolutivo} (GNAMPA, CUP E53C23001670001) and the research project {\em Studio di modelli nelle scienze della vita} (UniSS DM 737/2021 risorse 2022-2023).


\begin{thebibliography}{99}

\bibitem{ACF}
{\sc C.\ Anedda, F.\ Cuccu, S.\ Frassu,}
Steiner symmetry in the minimization of the first eigenvalue of a fractional eigenvalue problem with indefinite weight,
{\em Canad. J. Math.} {\bf 73} (2021) 970--992.

\bibitem{BCF}
{\sc C.\ Bjorland, L.\ Caffarelli, A.\ Figalli,}
Nonlocal tug-of-war and the infinity fractional Laplacian,
{\em Comm. Pure Appl. Math.} {\bf 65} (2012) 337--380.

\bibitem{B}
{\sc G.R.\ Burton,}
Rearrangements of functions, maximization of convex functionals and vortex rings,
{\em Math. Ann.} {\bf 276} (1987) 225--253.

\bibitem{B1}
{\sc G.R.\ Burton,}
Variational problems on classes of rearrangements and multiple configurations for steady vortices,
{\em Ann. Inst. Henri Poincar\'e} {\bf 6} (1989) 295--319.

\bibitem{BML}
{\sc G.R.\ Burton, J.B.\ McLeod,}
Maximisation and minimisation on classes of rearrangements,
{\em Proc. Roy. Soc. Edinburgh Sect. A} {\bf 119} (1991) 287--300.

\bibitem{C}
{\sc L.\ Caffarelli,}
Non-local diffusions, drifts and games,
in {\em Nonlinear Partial Differential Equations (Oslo 2010)},
Abel Symp. 7, Springer, Berlin (2012), 37--52.

\bibitem{CEP}
{\sc F.\ Cuccu, B.\ Emamizadeh, G.\ Porru,}
Optimization of the first eigenvalue in problems involving the $p$-Laplacian,
{\em Proc. Amer. Math. Soc.} {\bf 137} (2009) 1677--1687.

\bibitem{CPS}
{\sc F.\ Cuccu, G.\ Porru, S.\ Sakaguchi,}
Optimization problems on general classes of rearrangements,
{\em Nonlinear Analysis} {\bf 74} (2011) 5554--5565.

\bibitem{DPQ}
{\sc L.M.\ Del Pezzo, A.\ Quaas,}
A Hopf's lemma and a strong minimum principle for the fractional $p$-Laplacian,
{\em J. Differ. Equ.} {\bf 263} (2017) 765--778.

\bibitem{DNPV}
{\sc E.\ Di Nezza, G.\ Palatucci, E.\ Valdinoci,}
Hitchhiker's guide to the fractional Sobolev spaces,
{\em Bull. Sci. Math.} {\bf 136} (2012) 521--573.

\bibitem{FP}
{\sc G.\ Franzina, G.\ Palatucci,}
Fractional $p$-eigenvalues,
{\em Riv. Mat. Univ. Parma} {\bf 5} (2014) 373--386.

\bibitem{I}
{\sc A.\ Iannizzotto,}
Monotonicity of eigenvalues of the fractional $p$-Laplacian with singular weights,
{\em Topol. Methods Nonlinear Anal.} {\bf 61} (2023) 423--443.

\bibitem{ILPS}
{\sc A.\ Iannizzotto, S.\ Liu, K.\ Perera, M.\ Squassina,}
Existence results for fractional $p$-Laplacian problems via Morse theory,
{\em Adv. Calc. Var.} {\bf 9} (2016) 101--125.

\bibitem{IM}
{\sc A.\ Iannizzotto, S.\ Mosconi,}
Fine boundary regularity for the singular fractional $p$-Laplacian,
{\em J. Differential Equations} {\bf 412} (2024) 322--379.

\bibitem{IM1}
{\sc A.\ Iannizzotto, S.\ Mosconi,}
On a doubly sublinear fractional $p$-Laplacian equation,
preprint (arXiv:2409.03616v1).

\bibitem{IMP}
{\sc A.\ Iannizzotto, S.\ Mosconi, N.S.\ Papageorgiou,}
On the logistic equation for the fractional $p$-Laplacian,
{\em Math. Nachr.} {\bf 296} (2023) 1451--1468.

\bibitem{IM2}
{\sc A.\ Iannizzotto, D.\ Mugnai,}
Optimal solvability for the fractional $p$-Laplacian with Dirichlet conditions,
{\em Fract. Calc. Appl. Anal.} (2024) (doi.org/10.1007/s13540-024-00341-w).

\bibitem{IN}
{\sc H.\ Ishii, G.\ Nakamura,}
A class of integral equations and approximation of $p$-Laplace equations,
{\em Calc. Var. Partial Differential Equations} {\bf 37} (2010) 485--522.

\bibitem{KLP}
{\sc B.\ Kawohl, M.\ Lucia, S.\ Prashanth,}
Simplicity of the first eigenvalue for indefinite quasilinear problems,
{\em Adv. Differential Equations} {\bf 12} (2007) 407--434.

\bibitem{L}
{\sc G.\ Leoni,}
A first course in fractional Sobolev spaces,
American Mathematical Society, Providence (2023).

\bibitem{LL}
{\sc E.\ Lindgren, P.\ Lindqvist,}
Fractional eigenvalues,
{\em Calc. Var. Partial Differential Equations} {\bf 49} (2014) 795--826.

\bibitem{MBRS}
{\sc G.\ Molica Bisci, V.D.\ R\u{a}dulescu, R.\ Servadei,}
Variational methods for nonlocal fractional problems, Cambridge University Press, Cambridge (2016).

\bibitem{P}
{\sc G.\ Palatucci,}
The Dirichlet problem for the $p$-fractional Laplace equation,
{\em Nonlinear Anal.} {\bf 177} (2018) 699--732.

\bibitem{PV}
{\sc B.\ Pellacci, G.\ Verzini,}
Best dispersal strategies in spatially heterogeneous environments: optimization of the principal eigenvalue for indefinite fractional Neumann problems,
{\em J. Math. Biol.} {\bf 76} (2018) 1357--1386.

\bibitem{QHZ}
{\sc C.\ Qiu, Y.\ Huang, Y.\ Zhou,}
Optimization problems involving the fractional Laplacian,
{\em Electron. J. Differential Equations} {\bf 2016} (2016) art.\ 98.

\bibitem{R}
{\sc T.R.\ Rockafellar, } Convex analysis,
Princeton University Press, Princeton (1970).

\bibitem{RO}
{\sc X.\ Ros-Oton,}
Nonlocal elliptic equations in bounded domains: a survey,
{\em Publ. Mat.} {\bf 60} (2016) 3--26.

\end{thebibliography}
\end{document}